\documentclass[preprint,10pt]{elsarticle}
\usepackage{amsmath}

\usepackage{syntonly}

\newtheorem{theorem}{Theorem}
\newtheorem{lemma}{Lemma}

\newtheorem{remark}{Remark}
\newenvironment{proof}{{\bf Proof:}}{\hfill$\bull$\medskip}
\newcommand{\bull}{\vrule height 1.8ex width 1.0ex depth 0ex}

\def\XXint#1#2#3{{\setbox0=\hbox{$#1{#2#3}{\int}$ }
\vcenter{\hbox{$#2#3$ }}\kern-.5\wd0}}

\usepackage{graphicx}

\usepackage{threeparttable}
\usepackage{multirow}
\usepackage{appendix}
\usepackage{amssymb}
\usepackage{latexsym}
\usepackage{amsmath,amssymb}
\usepackage{amsfonts}
\usepackage{booktabs}
\usepackage{leftidx}
\usepackage{mathrsfs}
\usepackage{color}
 
   \def\XXint#1#2#3{{\setbox0=\hbox{$#1{#2#3}{\int}$}
   \vcenter{\hbox{$#2#3$}}\kern-.5\wd0}}

\journal{journal}
\begin{document}
\begin{frontmatter}


%
\title{\textbf{Computation of highly oscillatory Bessel transforms with algebraic singularities}\tnoteref{label1}}
\tnotetext[label1]{This work was supported by the Doctoral Scientific Research Foundation of Zhengzhou University of Light Industry
(No.~2015BSJJ071), and by the NSF of China (No.~11371376), and by the Innovation-Driven Project
and the Mathematics and Interdisciplinary Sciences Project of Central South University.}

\author{Zhenhua Xu\fnref{cor1a}}
\ead{xuzhenhua19860536@163.com}
\fntext[cor1a]{College of Mathematics and Information Science, Zhengzhou University of Light Industry, Zhengzhou, Henan 450002, China.}
\author{Shuhuang Xiang\fnref{cor2a}}
\ead{xiangsh@mail.csu.edu.cn}
\fntext[cor2a]{School of Mathematics and Statistics, Central South University, Changsha, Hunan 410083, China.}




\begin{abstract}
  In this paper, we consider the Clenshaw-Curtis-Filon method for the
  highly oscillatory Bessel transform $\int_0^1x^\alpha (1-x)^\beta f(x) J_{\nu}(\omega x)dx$, where
  $f$ is a smooth function on $[0, 1]$, and $\nu\geq0.$ The method is based on Fast Fourier Transform (FFT) and fast
  computation of the modified moments. We give a recurrence relation for the modified moments and
  present an efficient method for the evaluation of modified moments by using recurrence relation. Moreover, the corresponding
  error bound in inverse powers of  $N$ for this method for the integral is presented.
  Numerical examples are provided to support our analysis and show the efficiency and accuracy of the method.
\end{abstract}

\end{frontmatter}
\section{\textbf{Introduction}} \setcounter{theorem}{0} \setcounter{equation}{0}
\hspace{0.1cm}
 The fast computation of highly oscillatory Bessel transforms
\begin{align}\label{eq:a1}
I[f]=\int_0^1x^\alpha (1-x)^\beta f(x) J_{\nu}(\omega x)dx,\quad  \nu\geq 0
\end{align}
where $\alpha>-1, \beta>-1$, and $f$ is a sufficiently smooth function on $[0,1]$, $J_\nu(x)$ is the Bessel function
of the first kind and order $\nu$, plays an important role in many areas of science
and engineering, such as astronomy, optics, quantum mechanics,
seismology image processing, electromagnetic scattering (for example \cite{Arfken}, \cite{Bao}, \cite{Davis}, \cite{Huybrechs1}).
In most of the cases, the oscillatory integrals with Bessel kernels can not be evaluated analytically and
one has to resort to numerical methods. Particularly, for large $\omega$, the integrands
become highly oscillatory. Hence, it presents serious difficulties
in approximating the integral  by classical numerical methods like Simpson rule, Gaussian quadrature, etc.

In recent years, there has been tremendous
interest in developing numerical methods for the integral $\int_a^bf(x)J_{\nu}(\omega x)dx$, such
as Levin method \cite{Levin1, Levin2}, Levin-type method \cite{Olver1}, modified  Clenshaw-Curtis method \cite{Piessens},
generalized quadrature rule \cite{Evans, Evans1}, Filon-type
method \cite{Xiang3}, Gauss-Laguerre quadrature\cite{Chen, Chen2}.
To avoid the Runge phenomenon, a Clenshaw-Curtis-Filon-type method was
presented in  \cite{Xiang}. Since, the integrand in (\ref{eq:a1}) may have singularities
 at two end points,
these methods cannot be applied to evaluating the integral (\ref{eq:a1}) directly.
Recently, a Filon-type method based on a special Hermite interpolation polynomial at
Clenshaw-Curtis points was introduced in \cite{Chen3}. The key issue is the computation
of modified moments. However, the algorithm \cite{Chen3} to evaluate the
modified moments by transferring the Chebyshev interpolation polynomial
into power series of $x^k$ is quite unstable for
the number of nodes $N\geq 32$.

In this paper, we are concerned with the Clenshaw-Curtis-Filon method  for the computation of the highly
oscillatory Bessel transform (\ref{eq:a1}) based on fast computation of modified moments.
Moreover, this method can be applied to the Filon-type method based on Clenshaw-Curtis points given in \cite {Chen3}.

This paper is organized as follows.
In Section \ref{sect:2},
we describe the Clenshaw-Curtis-Filon method for the integral (\ref{eq:a1}).
Meanwhile, we deduce a recurrence relation for the modified moments, and present an efficient algorithm for the moments. Error
analysis about  $N$ for the presented method for the integral (\ref{eq:a1}) is derived in
Section \ref{sect:3}. We show that the method converges uniformly in $N$ for fixed $\omega$.
In Section \ref{sect:4}, several numerical examples are given to illustrate the accuracy and efficiency of the presented method.

\section{\textbf{Clenshaw-Curtis-Filon method for the integral
(\ref{eq:a1})}} \label{sect:2} \setcounter{equation}{0} \setcounter{theorem}{0}
\setcounter{equation}{0} \setcounter{lemma}{0}
\setcounter{proposition}{0} \setcounter{corollary}{0}

Polynomial interpolation is used as one of the basic means of approximation in most areas of numerical analysis.  To avoid
Runge phenomenon, we consider a special Lagrange interpolation polynomial at the Clenshaw-Curtis points, instead of equally spaced points.
Suppose that $f$ is absolutely continuous on $[0,1]$, and let
$P_{N}(x)$ be the interpolant of $f$ at the Clenshaw-Curtis points
$$c_i=\frac{1}{2}+\frac{1}{2} \cos\left(\frac{i\pi}{N}\right),\textrm{\qquad} i=0,1,\ldots,N,$$
defined by $P_{N}(x)=\sum_{i=0}^{N}b_i T_i^{\ast}(x)$,
where $T_i^*(x)$ denotes the shifted Chebyshev polynomial of the first kind on $[0,1]$, and the coefficients $b_i$ can be evaluated by
FFT \cite{Dahlquist, Davis1, Trefethen1, Trefethen2}. The
Clenshaw-Curtis-Filon  method (CCF) for (\ref{eq:a1}) is defined by
\begin{eqnarray}\label{eq:b1}
Q^{CCF}[f]=\int_0^1P_{N}(x) x^\alpha (1-x)^\beta J_\nu(\omega  x)dx =\sum_{k=0}^{N}b_k M(k,\nu,\omega),
\end{eqnarray}
where
\begin{eqnarray}\label{eq:b2}
M(k,\nu,\omega)=\int_0^1x^\alpha (1-x)^\beta  T_k^{\ast}(x) J_\nu(\omega  x)dx,
\end{eqnarray}
denote the modified moments.


\subsection{\textbf{A recurrence relation for $M(k,\nu,\omega)$}}\label{sect:2.1}
In the following, we present a recurrence relation for the modified moments $M(k,\nu,\omega)$.

\begin{theorem}\label{th:b1}
The sequence $M(k,\nu,\omega), k\geq 4$ satisfies the following recurrence relation:

\begin{equation}\label{eq:b3}
\begin{aligned}
\displaystyle &\frac{\omega^2}{16}M(k+4,\nu,\omega)+\Big[(k+3)^2-{\nu}^2-\frac{\omega^2}{4}+(\alpha+\beta)^2+(6+2k)(\alpha+\beta)\Big]M(k+2,\nu,\omega)\\
\displaystyle +&\Big[(4{\nu}^2+2k+4)-4(\alpha^2-\beta^2)-4k(\alpha-\beta)-8\alpha+12\beta\Big]M(k+1,\nu,\omega)\\
\displaystyle +&\Big[6(\alpha^2+\beta^2)+(4\alpha+12\beta-4\alpha\beta)-2k^2+6-6 {\nu^2}+\frac{3\omega^2}{8}\Big]M( {k},\nu,\omega)\\
\displaystyle +&\Big[(4{\nu}^2-2k+4)-4(\alpha^2-\beta^2)-4k(\alpha-\beta)-8\alpha+12\beta\Big]M(k-1,\nu,\omega)\\
\displaystyle +&\Big[(k-3)^2-{\nu}^2-\frac{\omega^2}{4}+(\alpha+\beta)^2+(6-2k)(\alpha+\beta)\Big]M(k-2,\nu,\omega)\\
\displaystyle +&\frac{\omega^2}{16}M(k-4,\nu,\omega)=0.
\end{aligned}
\end{equation}
\end{theorem}
\begin{proof}
Firstly, we rewrite the modified moments as
\begin{eqnarray}
M(k,\nu,\omega)=\frac{1}{2^{\alpha+\beta+1}}\int_{-1}^1(1+x)^\alpha(1-x)^\beta T_k(x)J_\nu\Big(\frac{1+x}{2}\omega\Big )dx,
\end{eqnarray}
where $T_k(x)$ is the Chebyshev polynomial of degree $k$.

Let
\begin{eqnarray}\label{eq:b4}
K_1=4\int_{-1}^1(1+x)^\alpha(1-x)^\beta (1-x)^2(1+x)^2 T_k(x)\bigg[J_\nu\Big(\frac{1+x}{2}\omega \Big)\bigg]^{\prime\prime}dx,
\end{eqnarray}
\begin{eqnarray}\label{eq:b5}
K_2=4\int_{-1}^1(1+x)^\alpha(1-x)^\beta (1-x)^2(1+x) T_k(x)\bigg[J_\nu\Big(\frac{1+x}{2}\omega \Big)\bigg]^{\prime}dx,
\end{eqnarray}
and
\begin{eqnarray}\label{eq:b6}
K_3=4\int_{-1}^1(1+x)^\alpha(1-x)^\beta (1-x)^2\Big(\nu^2-\frac{(1+x)^2\omega^2}{4}\Big)T_k(x)J_\nu\Big(\frac{1+x}{2}\omega\Big)dx.
\end{eqnarray}
According to the fact that \cite{Abram}
\begin{eqnarray}\label{eq:bessel}
x^2\frac{d^2 J_\nu(x)}{dx^2}+x\frac{d J_\nu(x)}{dx}+(x^2-\nu^2)J_\nu(x)=0,
\end{eqnarray}
we can easily obtain the equality
\begin{eqnarray}\label{eq:b7}
K_1+K_2-K_3=0.
\end{eqnarray}

Using the identity that \cite{Mason}
\begin{eqnarray}\nonumber
x^mT_n(x)=2^{-m}\sum_{j=0}^m\binom{m}{j}T_{n+m-2j}(x)
\end{eqnarray}
and integration by parts, we have
\begin{equation}\label{eq:b8}
\begin{aligned}
\displaystyle K_1=&\Big[(\alpha^2+\beta^2+k^2)+2(\alpha\beta+\alpha k+\beta k)+7(\alpha+\beta+k)+12\Big]M(k+2,\nu,\omega)\\
\displaystyle &+\Big[4(\beta^2-\alpha^2)+(4k+12)(\beta-\alpha)\Big]M(k+1,\nu,\omega)\\
\displaystyle &+\Big[6(\alpha^2+\beta^2)+10(\alpha+\beta)-4\alpha\beta-2k^2+8\Big]M(k,\nu,\omega)\\
\displaystyle &+\Big[4(\beta^2-\alpha^2)+(12-4k)(\beta-\alpha)\Big]M(k-1,\nu,\omega)\\
\displaystyle &+\Big[(\alpha^2+\beta^2+k^2)+2(\alpha\beta-\alpha k-\beta k)+7(\alpha+\beta-k)+12\Big]M(k-2,\nu,\omega),
\end{aligned}
\end{equation}

\begin{equation}\label{eq:b9}
\begin{aligned}
\displaystyle K_2=&-\Big\{(\alpha+\beta+k+3)M(k+2,\nu,\omega)-(4+4\alpha+2k)M(k+1,\nu,\omega)+(6\alpha-2\beta+2)\\
\displaystyle &M(k,\nu,\omega)-(4+4\alpha-2k)M(k-1,\nu,\omega)+(\alpha+\beta-k+3)M(k-2,\nu,\omega)\Big\},
\end{aligned}
\end{equation}
and

\begin{equation}\label{eq:b10}
\begin{aligned}
\displaystyle K_3=&-\frac{1}{16}\bigg\{\omega^2M(k+4,\nu,\omega)-(4\omega^2+16\nu^2)M(k+2,\nu,\omega)+64\nu^2M(k+1,\nu,\omega)+\\
\displaystyle &(6\omega^2-96\nu^2)M(k,\nu,\omega)+64\nu^2M(k-1,\nu,\omega)-(4\omega^2+16\nu^2)M(k-2,\nu,\omega)\\
\displaystyle &+\omega^2M(k-4,\nu,\omega)\bigg\}.
\end{aligned}
\end{equation}
A combination of (\ref{eq:b7})-(\ref{eq:b10}) gives the desired result.
\end{proof}

\subsection{\textbf{Fast computations of the modified moments}}\label{sect:2.2}

Now, we turn to the fast computations of the modified moments via recurrence relation (\ref{eq:b3}).
Because of the symmetry of the recurrence relation of
the Chebyshev polynomials $T_j(x)$, it is convenient to define
$T_{-j}(x) =T_j(x), j=1, 2, \ldots$, then it holds that $M(-j,\nu,\omega)=M(j,\nu,\omega)$. It can be
shown that (\ref{eq:b3}) is not only valid for $k\geq 4$, but also for all integers of $k$. However, both
forward and backward recursion are asymptotically unstable \cite{Piessens}.
Fortunately, the instability is less pronounced if $\omega$ is large, and practical experiments  show
that the modified moments $M(j,\nu,\omega)$ can be computed accurately using forward recursion
if $k\leq\omega/2$. For the case $k\geq\frac{\omega}{2}$,
forward recursion is no longer feasible, since the loss of significant figures increases.
In this case, the moments can
be computed by Oliver's algorithm \cite{Oliver} with six starting
moments and two end moments. The end moments can be approximated by
converting  $M(k,\nu,\omega)$  into the form of Fourier integral and
using asymptotic expansion in \cite{Erdelyi}.

Using the explicit expression of the shifted Chebyshev polynomials \cite{Mason}
\begin{align*}
T_k^{\ast}(x)=T_{2k}(\sqrt{x})=\sum_{j=0}^kc_j^{(2k)} x^{k-j},
\end{align*}
where
\begin{eqnarray*}
c_j^{(2k)}={(-1)}^j2^{2k-2j-1}\Big[2\binom{2k-j}{j}-\binom{2k-j-1}{j}\Big],
\end{eqnarray*}
the first few modified moments can be evaluated efficiently as follows:
\begin{eqnarray}\label{eq:b11}
M(k,\nu,\omega)=\sum_{j=0}^kc_j^{(2k)}I(\alpha+k-j,\beta,\nu,\omega),
\end{eqnarray}
and $I(\alpha+k-j,\beta,\nu,\omega)=\int_0^1x^{\alpha+k-j}(1-x)^\beta J_\nu(\omega x)dx$ can be computed by the formula \cite[p. 681] {Gradshteyn}

\begin{equation}\label{eq:b12}
\begin{aligned}
\displaystyle &\int_0^1x^a(1-x)^bJ_\nu(\omega x)dx\\
\displaystyle =&\frac{\Gamma(b+1)\Gamma(a+\nu+1)\leftidx{_2}{F}{_3}(\frac{a+\nu+1}{2},\frac{a+\nu+2}{2};
\nu+1,\frac{a+b+\nu+2}{2},\frac{a+b+\nu+3}{2};-\frac{\omega^2}{4})}{2^\nu\omega^{-\nu}\Gamma(\nu+1)\Gamma(a+b+\nu+2)},
\end{aligned}
\end{equation}
where $\Re(a+\nu)>-1, \Re(b)>-1$, and $\leftidx{_2}{F}{_3}(\mu_1,\mu_2; \nu_1, \nu_2, \nu_3; z)$
denotes a class of generalized hypergeometric function.

\begin{remark}
Since $Y_\nu(x)$ and $J_\nu(x)$ are the solutions of the same differential equation, the integrals
$\int_0^1 x^\alpha (1-x)^\beta T_k^{\ast}(x)Y_\nu(\omega  x)dx, k=0, 1, \ldots ,$ also satisfy the same
recurrence relation (\ref{eq:b3}). Moreover, the function $Y_\nu(x)$ can be expressed by the equation \cite[p. 219] {Bateman}
\begin{eqnarray}\label{eq:re1}
\nonumber Y_{\nu}(x)= G^{2,0}_{1,3}\left(\begin{array}{c}-\frac{\nu}{2}-\frac{1}{2}\\ -\frac{\nu}{2}, \frac{\nu}{2}, -\frac{\nu}{2}-\frac{1}{2} \end{array}\Bigg|\frac{x^2}{4} \right).
\end{eqnarray}
According to the identity that \cite{Meijer}
 \begin{eqnarray}\label{eq:re2}
\nonumber & &\int_0^x t^{\alpha-1}(x-t)^{\beta-1}G^{m,n}_{p,q}\left(\begin{array}{c} a_1\ldots a_n, a_{n+1} \ldots a_p \\ b_1\ldots b_m, b_{m+1} \ldots b_q \end{array}\Bigg|\omega t^{l}\right)dt \\
&=&\frac{l^{-\beta}\Gamma(\beta)}{x^{1-\alpha-\beta}}G^{m,n+l}_{p+l,q+l}\left(\begin{array}{c} \frac{1-\alpha}{l},\ldots, \frac{l-\alpha}{l}, a_1\ldots a_n, a_{n+1} \ldots a_p \\ b_1\ldots b_m, b_{m+1} \ldots b_q, \frac{1-\alpha-\beta}{l},\ldots, \frac{l-\alpha-\beta}{l} \end{array}\Bigg|\omega x^{l}\right),
\end{eqnarray}
the first several starting values of the modified moments $\int_0^1 x^\alpha (1-x)^\beta T_k^{\ast}(x)
Y_\nu(\omega  x)dx, k=1, 2, \ldots$ can be evaluated by the following
formula
 \begin{eqnarray}\label{eq:re3}
\int_0^1 x^a (1-x)^b
Y_{\nu}(\omega x)dx=\frac{\Gamma(b+1)}{2^{b+1}}G^{2,2}_{3,5}\left(\begin{array}{c}-\frac{a}{2}, \frac{1-a}{2}, -\frac{\nu
}{2}-\frac{1}{2}\\ -\frac{\nu
}{2}, \frac{\nu
}{2}, -\frac{\nu
}{2}-\frac{1}{2}, -\frac{a+b+1}{2}, -\frac{a+b}{2}\end{array}\Bigg|\frac{1}{4}\omega^2\right)
\end{eqnarray}
by a similar way to the modified moments $M(k,\nu,\omega)$. As this idea is tangential
to the topic of this paper, we will not study it further.
\end{remark}
\section{\textbf{ Error analysis and uniform convergence for the Clenshaw-Curtis-Filon
method}}  \label{sect:3}\setcounter{equation}{0}\setcounter{theorem}{0}
\setcounter{lemma}{0}

In practical problems the frequency
$\omega$ is always fixed. To guarantee the convergence of the
method for the fixed $\omega$ with respect to the number
of quadrature nodes, we mainly focus
on the error bound about $N$ for fixed $\omega$ in this section.
We first introduce some lemmas.



\begin{lemma}(\cite{Erdelyi})\label{le:1}
If $0<\lambda, \mu\leq1$, and $\phi(t)$ is $N$ times differentiable for $a \leq t\leq b$, then
\begin{equation}\label{eq:b13}
\int_a^b(t-a)^{\lambda-1}(b-t)^{\mu-1} e^{i r t}\phi(t)dt=B_N(r)-A_N(r)+O(r^{-N}),
\end{equation}
where
\begin{equation}\nonumber
\begin{aligned}
\displaystyle
A_N(r)&=\sum_{n=0}^{N-1}\frac{\Gamma({n+\lambda})}{n!}e^{i\pi(n+\lambda-2)/2}\omega^{-n-\lambda}
e^{ir a}\bigg[\frac{d^n}{dt^n}\big\{(b-t)^{\mu-1}\phi(t)\big\}\bigg]_{t=a},\\
\displaystyle B_N(r)&=\sum_{n=0}^{N-1}\frac{\Gamma({n+\mu})}{n!}e^{i\pi(n-\mu)/2}\omega^{-n-\mu}
e^{ir b}\bigg[\frac{d^n}{dt^n}\big\{(t-a)^{\lambda-1}\phi(t)\big\}\bigg]_{t=b}.
\end{aligned}
\end{equation}
\end{lemma}

%
%

\begin{lemma}\label{le:2}
Suppose that $f\in C^{\kappa+1}[a,b]$, for each $\alpha>-1, \beta>-1$, and $r \gg1$, it holds that
\begin{equation}\label{eq:b16}
\int_a^b(x-a)^\alpha(b-x)^\beta f(x) e^{ir x}dx=O\Big(r^{-1-\min\big\{\alpha, \beta\big\}}\Big),
\end{equation}
where $\kappa=\Big\lceil\min\big\{\alpha,\beta\big\}\Big\rceil$, and $\big \lceil z \big\rceil$ denotes
the smallest integer not less than $z$.
\end{lemma}

\begin{proof}
We only prove (\ref{eq:b16}) for the case $\alpha\leq\beta$, and the similar proof can be applied to
the case $\alpha > \beta$.

Assume that $\alpha=\lceil \alpha \rceil +k_1=N_1+k_1$ and $\beta=\lceil \beta \rceil +k_2=N_2+k_2$, we have
$N_1\leq N_2$ and $ -1<k_1, k_2\leq 0$. Then, it holds that
\begin{equation}\label{eq:b16a}
\begin{aligned}
\displaystyle
&\int_a^b(x-a)^\alpha(b-x)^\beta f(x) e^{ir x}dx\\
\displaystyle =&\int_a^b e^{ir x}(x-a)^{k_1}(b-x)^{k_2}(x-a)^{N_1}(b-x)^{N_2} f(x) dx\\
\displaystyle =&\int_a^b e^{ir x}(x-a)^{k_1}(b-x)^{k_2}F(x) dx,
\end{aligned}
\end{equation}
where $F(x)=(x-a)^{N_1}(b-x)^{N_2} f(x)$.

If $N_1=0$, it yields the desired result by
applying the Lemma \ref{le:2} to the first identity of equation (\ref{eq:b16a}) directly.

If $N_1\geq1$, it can be shown that
\begin{equation}\label{eq:b17}
\begin{aligned}
\displaystyle
\bigg[\frac{d^j}{dx^j}\big\{(b-x)^{k_2}F(x)\big\}\bigg]_{x=a}=0,\quad j=0,\ldots,N_1-1,
\end{aligned}
\end{equation}
and
\begin{equation}\label{eq:b18}
\begin{aligned}
\displaystyle
\bigg[\frac{d^j}{dx^j}\big\{(x-a)^{k_1}F(x)\big\}\bigg]_{x=b}=0,\quad j=0,\ldots,N_2-1.
\end{aligned}
\end{equation}
According to Lemma \ref{le:1}, a combination of (\ref{eq:b16a})-(\ref{eq:b18}) yields that
\begin{equation}\nonumber
\int_a^b(x-a)^\alpha(b-x)^\beta f(x) e^{ir x}dx=O\Big(r^{-N_1-k_1-1}\Big)=O\Big(r^{-1-\min\big\{\alpha, \beta\big\}}\Big).
\end{equation}
This completes the proof.
\end{proof}

\begin{lemma}\label{le:3}
For each $j\geq 1, \alpha>-1,\beta>-1$ and fixed $\omega$, it is true that
\begin{equation}\label{eq:b22}
\int_0^1x^\alpha (1-x)^\beta T_j^{\ast}(x)J_{\nu} (\omega x)dx=\left\{ \begin{aligned}
        &O(j^{-2}),~~~~~~~~~~~~~~~~~~\alpha=\beta=-\frac{1}{2} ,\\
                & O\Big(j^{-\min\big\{2, 2+2\alpha\big\}}\Big), \quad\beta=-\frac{1}{2},\alpha> -\frac{1}{2},\\
                 &O\Big(j^{-\min\big\{2, 2+2\beta\big\}}\Big), \quad\alpha=-\frac{1}{2}, \beta>-\frac{1}{2}, \\
                  &O\Big(j^{-2-2\min\big\{\alpha,\beta\big\}}\Big),\quad \mbox{otherwise}.
                          \end{aligned} \right.
\end{equation}
\end{lemma}
\begin{proof}
By changing the variables $x=\frac{1+t}{2}$ and $t=\cos(\theta)$, it yields that
\begin{equation}\label{eq:b23}
\begin{aligned}
\displaystyle \int_0^1x^\alpha (1-x)^\beta T_j^{\ast}J_{\nu} (\omega x)dx&=\frac{1}{2}\int_0^\pi\sin(\theta)\cos^{2\alpha}(\theta/2)
\sin^{2\beta}(\theta/2)\cos(j\theta)J_\nu(\omega\cos^{2}(\theta/2))d\theta\\
\displaystyle &=2(-1)^j\int_0^{\frac{\pi}{2}}\sin^{2\alpha+1}(\theta)\cos^{2\beta+1}(\theta)\cos(2j\theta)J_\nu(\omega \sin^2(\theta))d\theta\\
\displaystyle &=\frac{(-1)^j}{j}\int_0^{\frac{\pi}{2}}\sin^{2\alpha+1}(\theta)\cos^{2\beta+1}(\theta)J_\nu(\omega \sin^2(\theta))d\sin(2j\theta).
\end{aligned}
\end{equation}

For the case $\alpha=\beta=-\frac{1}{2}$, we easily derive the first identity in (\ref{eq:b22})
by integrating (\ref{eq:b23}) by parts two times.

For the case $\beta=-\frac{1}{2},\alpha> -\frac{1}{2}$ or $\alpha=-\frac{1}{2}, \beta>-\frac{1}{2}$, we rewrite the
integral (\ref{eq:b23}) as
\begin{equation}\label{eq:b24}
\begin{aligned}
\displaystyle &\int_0^1x^\alpha (1-x)^\beta T_j^{\ast}(x)J_{\nu} (\omega x)dx\\
\displaystyle =&2(-1)^j\int_0^{\frac{\pi}{2}}\theta^{2\alpha+1}(\pi/2-\theta)^{2\beta+1}
\Big(\frac{\sin(\theta)}{\theta}\Big)^{2\alpha+1}\Big(\frac{\cos(\theta)}{\pi/2-\theta}\Big)^{2\beta+1}
J_\nu(\omega \sin^2(\theta))\cos(2j\theta)d\theta.
\end{aligned}
\end{equation}
By using integration by parts one time for the integral (\ref{eq:b24}) and Lemma \ref{le:2}, we can easily get the
the second and third identities in (\ref{eq:b22}).

For other cases, the fourth identity in (\ref{eq:b22}) can be obtained by applying the Lemma \ref{le:2}
to integral (\ref{eq:b24}) directly.
\end{proof}

\begin{theorem}\label{th:2}
Suppose that $f(x)$ has an absolutely continuous $(k-1)$st derivative $f^{(k-1)}$ on
$[0,1]$ and a $k$th derivative $f^{(k)}$ of bounded variation $V_k$ for some $k\geq1$.
Then, for each $\alpha>-1, \beta>-1, N\geq k+1$ and fixed $\omega$, the error
bound about $N$ for the Clenshaw-Curtis-Filon method
for the integral (\ref{eq:a1}) satisfies
\begin{equation}\label{eq:b25}
\begin{aligned}
\displaystyle I[f]-Q^{CCF}[f]=\left\{ \begin{aligned}
        &O\Big(N^{-2\min\big\{\alpha,\beta\big\}-k-2}\Big),\quad -1<\min\Big\{\alpha,\beta\Big\}<-\frac{1}{2}, \\
                & O\big(N^{-k-1}\big),~~~~~~~~~~~~~~~~~~\min\Big\{\alpha,\beta\Big\}\geq-\frac{1}{2}.
                          \end{aligned} \right.
\end{aligned}
\end{equation}
\end{theorem}

\begin{proof}
According to the ideas of \cite{Xiang4, Xiang5}, the Chebyshev series for the function $f(x)$ can be expressed as  \cite[pp. 165] {Rivlin}
\begin{equation}\label{eq:b26}
f(x)=\sum^{\infty}_{j=0}{'}a_jT^\ast_j(x), \quad a_j=\frac{\pi}{2}\int_{-1}^1\frac{f(\frac{1+x}{2})T_j(x)}{\sqrt{1-x^2}}dx,
\end{equation}
where the prime indicates that the term with $j=0$ is multiplied by $1/2$.
Also, the coefficients $a_j$ satisfy \cite{Trefethen1, Trefethen2}
\begin{equation}\label{eq:b27}
|a_j|\leq \frac{2V_k}{\pi j (j-1)\cdot\cdot\cdot(j-k)}.
\end{equation}

For each $j=0,1,\ldots, N$, $p=1,2,\ldots$, from the property of Chebyshev polynomials \cite[p. 67] {Fox}, we can easily get
the aliasing errors for the integration on Chebyshev polynomials
\begin{equation}\label{eq:b28}
\begin{aligned}
\displaystyle P_N(T^\ast_{pN+j})=\left\{ \begin{aligned}
        &T^\ast_{N-j}, \quad \mbox{if \emph{p} is odd},\\
                & T^\ast_j, ~~~~~~~ \mbox{if \emph{p} is even},
                          \end{aligned} \right.
\end{aligned}
\end{equation}
and
\begin{equation}\label{eq:b29}
\begin{aligned}
\displaystyle Q^{CCF}[T^\ast_{pN+j}]=\left\{ \begin{aligned}
        &I[T^\ast_{N-j}], \quad \mbox{if \emph{p} is odd},\\
                & I[T^\ast_j], ~~~~~~~ \mbox{if \emph{p} is even}.
                          \end{aligned} \right.
\end{aligned}
\end{equation}



For the case $\min\big\{\alpha,\beta\big\}\geq -1/2$, according to Lemma \ref{le:3}, there
exists two constants $C$ and $\sigma> -1/2$ such that $|I[T^\ast_j]|\leq C j^{-2-2\sigma}$.
Thus, we have the following estimate:
\begin{equation}\label{eq:b31}
\begin{aligned}
\displaystyle &|I[f]-Q^{CCF}[f]|\\
\displaystyle \leq& \sum_{m=N+1}^\infty |a_m||I[T^\ast_m]-Q^{CCF}[T^\ast_m]|\\
\displaystyle =&\sum_{p=1}^\infty\sum_{j=1}^N\Big(|a_{2pN+j}||I[T^\ast_{2pN+j}]-I[T^\ast_j]|+|a_{(2p-1)N+j}||I[T^\ast_{(2p-1)N+j}]-I[T^\ast_{N-j}]|\Big)\\
\displaystyle \leq&\sum_{p=1}^\infty\sum_{j=1}^N\Bigg(\frac{2V_k\Big(\frac{C}{(2pN+j)^{2+2\sigma}}+\frac{C}{j^{2+2\sigma}}\Big)}{\pi(2pN+j)(2pN+j-1)\cdot\cdot\cdot(2pN+j-k)}\\
\displaystyle &+\frac{2V_k\Big(\frac{C}{((2p-1)N+j)^{2+2\sigma}}+\frac{C}{(N-j)^{2+2\sigma}}\Big)}{\pi((2p-1)N+j)((2p-1)N+j-1)\cdot\cdot\cdot((2p-1)N+j-k)}\Bigg)\\
\displaystyle <&\frac{2CV_k}{\pi N(N-1)\cdot\cdot\cdot(N-k)}\sum_{p=1}^\infty\Big(\frac{\hat{C}}{p^{k+2+2\sigma}N^{1+2\sigma}}+\frac{\tilde{C}}{{p^{k+1}}}\Big)\\
\displaystyle =&\frac{2CV_k}{\pi N(N-1)\cdot\cdot\cdot(N-k)}\Big(\frac{\zeta(k+2+2\sigma)\hat{C}}{N^{1+2\sigma}}+{\tilde{C}\zeta(k+1)}\Big),
\end{aligned}
\end{equation}
where $\zeta(k)$ is the zeta function defined by $\zeta(k)=\sum_{p=1}^\infty \frac{1}{p^k}$, and we use the following
estimates
\begin{equation}
\begin{aligned}\nonumber
\displaystyle \sum_{j=1}^N\frac{1}{(2pN+j)^{2+2\sigma}}&<\int_0^N\frac{1}{(2pN+x)^{2+2\sigma}}dx=\frac{1+2\sigma}{(2pN)^{1+2\sigma}}-\frac{1+2\sigma}{(2pN+N)^{1+2\sigma}}\\
\displaystyle &<\frac{1+2\sigma}{(2pN)^{1+2\sigma}}-\frac{1+2\sigma}{(4pN)^{1+2\sigma}}=\Big(\frac{1+2\sigma}{2^{1+2\sigma}}-\frac{1+2\sigma}{4^{1+2\sigma}}\Big)\frac{1}{(pN)^{1+2\sigma}}\\
\displaystyle &=\hat{C}\frac{1}{(pN)^{1+2\sigma}},
\end{aligned}
\end{equation}
and
\begin{equation}
\nonumber \sum_{j=1}^N\frac{1}{((2p-1)N+j)^2}<\frac{1}{(2p-1)^2N}, \quad \sum_{j=1}^N\frac{1}{j^{2+2\sigma}}<\sum_{j=1}^\infty\frac{1}{j^{2+2\sigma}}=\tilde{C}.
\end{equation}
Then the second  identity in (\ref{eq:b25}) holds.

By a similar way, we can easily
derive the first identity in (\ref{eq:b25}) for the case $-1<\min\big\{\alpha,\beta\big\}<-\frac{1}{2}$.
\end{proof}

Theorem \ref{th:2} shows that Clenshaw-Curtis-Filon method (\ref{eq:b1}) for the integral
(\ref{eq:a1}) converges uniformly in $N$ for fixed $\omega$.

\begin{remark}\label{re:2}
From the proof of the Theorem \ref{th:2}, we can see that if $f\in X^k$, equation (\ref{eq:b25}) also holds, where $X^k$ denotes the spaces of
the functions whose Chebyshev coefficients decay asymptotically as $a_j=O(j^{-k-1})$ for some positive $k$. Additionally, let ${Q}^{C-C-F}$ denote the  Clenshaw-Curtis-Filon method
for the integral $\bar{I}[f]=\int_a^b f(x)(x-a)^\alpha(b-x)^\beta e^{i\omega x}dx$ in \cite{Kang}, by a similar proof of
the Theorem \ref{th:2}, we can show that if $f$ satisfies the condition of Theorem \ref{th:2} or $f\in X^k$, for the fixed $\omega$, there also holds
\begin{equation}\label{eq:b25b}
\begin{aligned}
\displaystyle \bar{I}[f]-{Q}^{C-C-F}[f]=\left\{ \begin{aligned}
        &O\Big(N^{-2\min\big\{\alpha,\beta\big\}-k-2}\Big),\quad -1<\min\Big\{\alpha,\beta\Big\}<-\frac{1}{2}, \\
                & O\big(N^{-k-1}\big),~~~~~~~~~~~~~~~~~~\min\Big\{\alpha,\beta\Big\}\geq-\frac{1}{2}.
                          \end{aligned} \right.
\end{aligned}
\end{equation}
This result will be illustrated by several examples in Section \ref{sect:4} (see Figures \ref{fig:2}-\ref{fig:2a}).

\end{remark}

\begin{remark}\label{re:3}
From Theorem 2.2 in \cite{Chen3}, we can see that the error bound about $\omega$ for fixed $N$ for the
Clenshaw-Curtis-Filon method for the integral (\ref{eq:a1}) satisfies
\begin{equation*}
\big|I[f]-Q^{CCF}[f]\big|=O(\omega^{-\tau}), \quad \omega\gg1,
\end{equation*}
where $\tau=\min\Big\{\alpha+2,\beta+\frac{5}{2}\Big\}$.
Moreover, based on efficient evaluation of the modified moments, this method can be
applied to the Filon-type method based on Clenshaw-Curtis points \cite{Chen3}. It should be
noted that
the algorithm \cite{Chen3} to evaluate the
modified moments by transferring the Chebyshev interpolation polynomial
into power series of $x^k$ is quite unstable for
the number of nodes $N\geq 32$.
\end{remark}
\section{Numerical example}\label{sect:4}
\setcounter{theorem}{0}   \setcounter{equation}{0} \setcounter{lemma}{0} \setcounter{proposition}{0}
\setcounter{corollary}{0}

In this section, we will test several
numerical examples to illustrate the efficiency and accuracy
of the method (\ref{eq:b1}). The
values assumed to be accurate are computed in the {\sc Maple
14} using the 32 decimal digits precision arithmetic. The
experiments are performed by using R2012a version of the {\sc Matlab} system.

{\bf Example 1.} Consider the following integrals
\begin{equation}\label{eq:d1}
I_1[f]=\int_0^1|x-0.5|^k x^\alpha(1-x)^\beta J_0(\omega x)dx,
\end{equation}
and
\begin{equation}\label{eq:d1a}
I_2[f]=\int_0^1|x-0.5|^k x^\alpha(1-x)^\beta e^{i\omega x}dx,
\end{equation}
by Clenshaw-Curtis-Filon method, where $f(x)=|x-0.5|^k$ with $k=1$ and $3$,
$\alpha=0.2, \beta=0.4$. According to Theorem \ref{th:2} and Remark \ref{re:2}, it
yields the overall estimate $O(N^{-k-1})$
for the absolute error. Figures \ref{fig:1}-\ref{fig:2} illustrate the convergence rates
for $(N+1)$-point Clenshaw-Curtis-Filon method for the integrals
(\ref{eq:d1}) and (\ref{eq:d1a}).
As can be seen,  the asymptotic order on $N$ for fixed $\omega$ in Theorem \ref{th:2} is attainable
for the function $f(x)=|x-0.5|^k$ of limited regularity.

{\bf Example 2.} Consider the following integrals
\begin{equation}\label{eq:d1b}
I_3[f]=\int_0^1 (1-x^2)^{0.8} x^\alpha(1-x)^\beta J_0(\omega x)dx,
\end{equation}
and
\begin{equation}\label{eq:d1c}
I_4[f]=\int_0^1(1-x^2)^{0.8} x^\alpha(1-x)^\beta e^{i\omega x}dx,
\end{equation}
by Clenshaw-Curtis-Filon method, where $f(x)=(1-x^2)^{0.8}$ with
$\alpha=-0.8, \beta=-0.9$. From  \cite{Wang}, we see that $f \in X^{1.6}$. According to Theorem \ref{th:2} and Remark \ref{re:2}, it
yields the overall estimate $O(N^{-3.6-2\min\{\alpha,\beta\}})$
for the absolute error. Figure \ref{fig:2a} illustrates the convergence rates
for $(N+1)$-point Clenshaw-Curtis-Filon method for the integrals
(\ref{eq:d1b}) and (\ref{eq:d1c}).
As can be seen,  the asymptotic order on $N$ for fixed $\omega$ in Theorem \ref{th:2} is attainable
for the function $f(x)=(1-x^2)^{0.8}$.

\begin{figure}[hbtp]
\begin{center}
\includegraphics[scale=0.42]{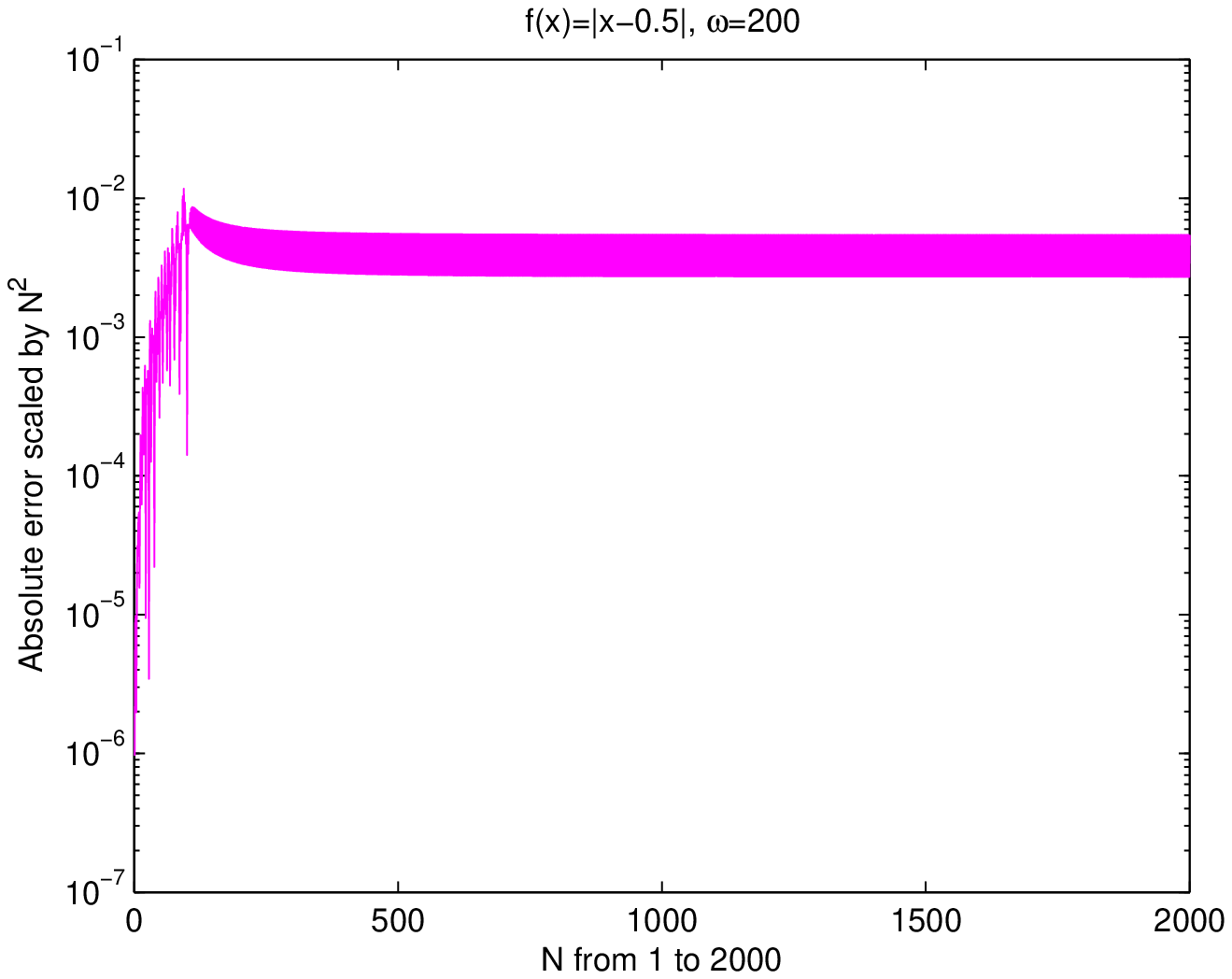}
\includegraphics[scale=0.42]{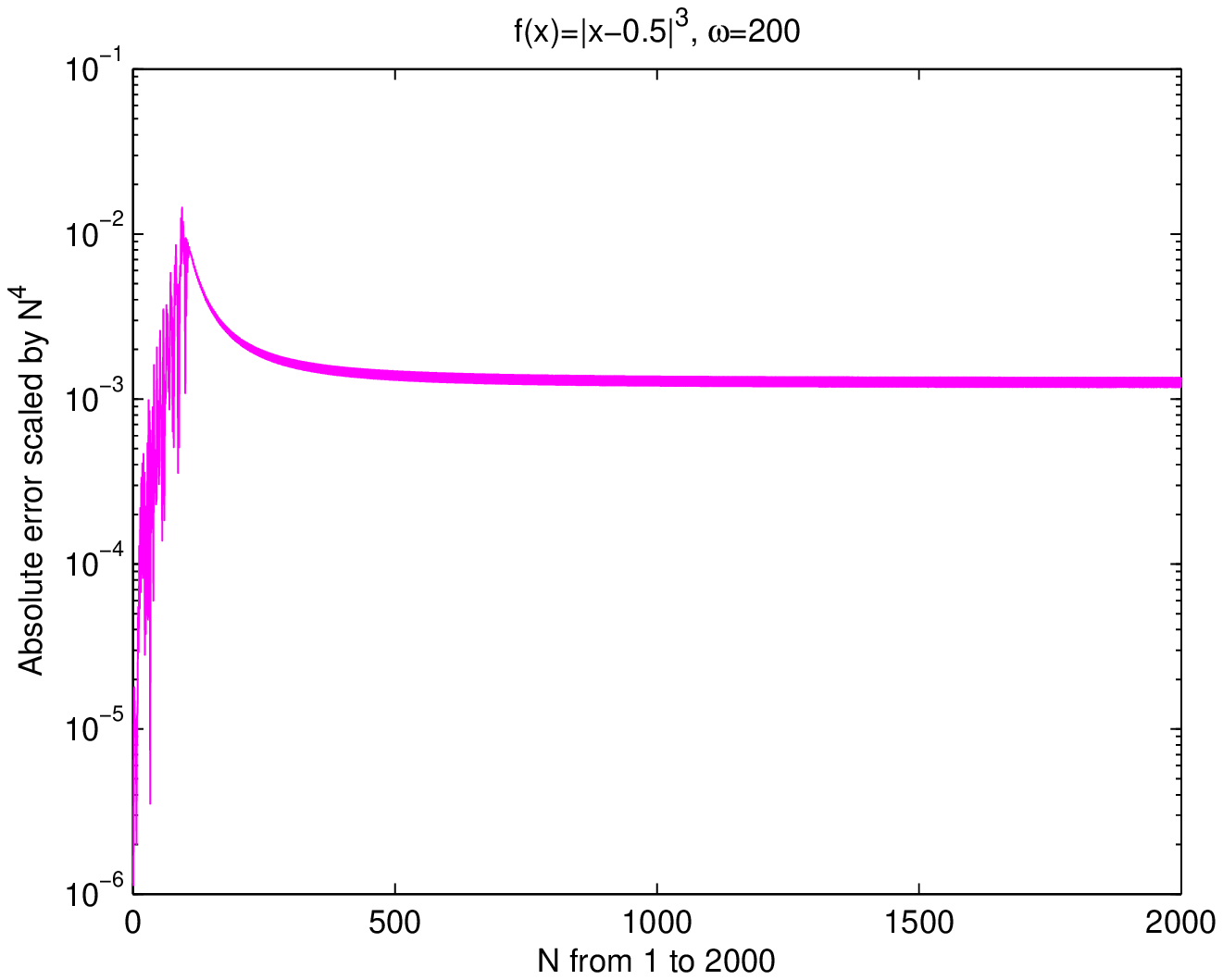}
\caption{Absolute  errors scaled by
$N^2$ (left) and $ N^4$ (right) for the integral (\ref{eq:d1})
for the Clenshaw-Curtis-Filon method with
$\omega=200$, $\alpha=0.2, \beta=0.4$, $N$ from $1$ to $2000$.}
\label{fig:1}
\end{center}
\end{figure}
\begin{figure}[hbtp]
\begin{center}
\includegraphics[scale=0.42]{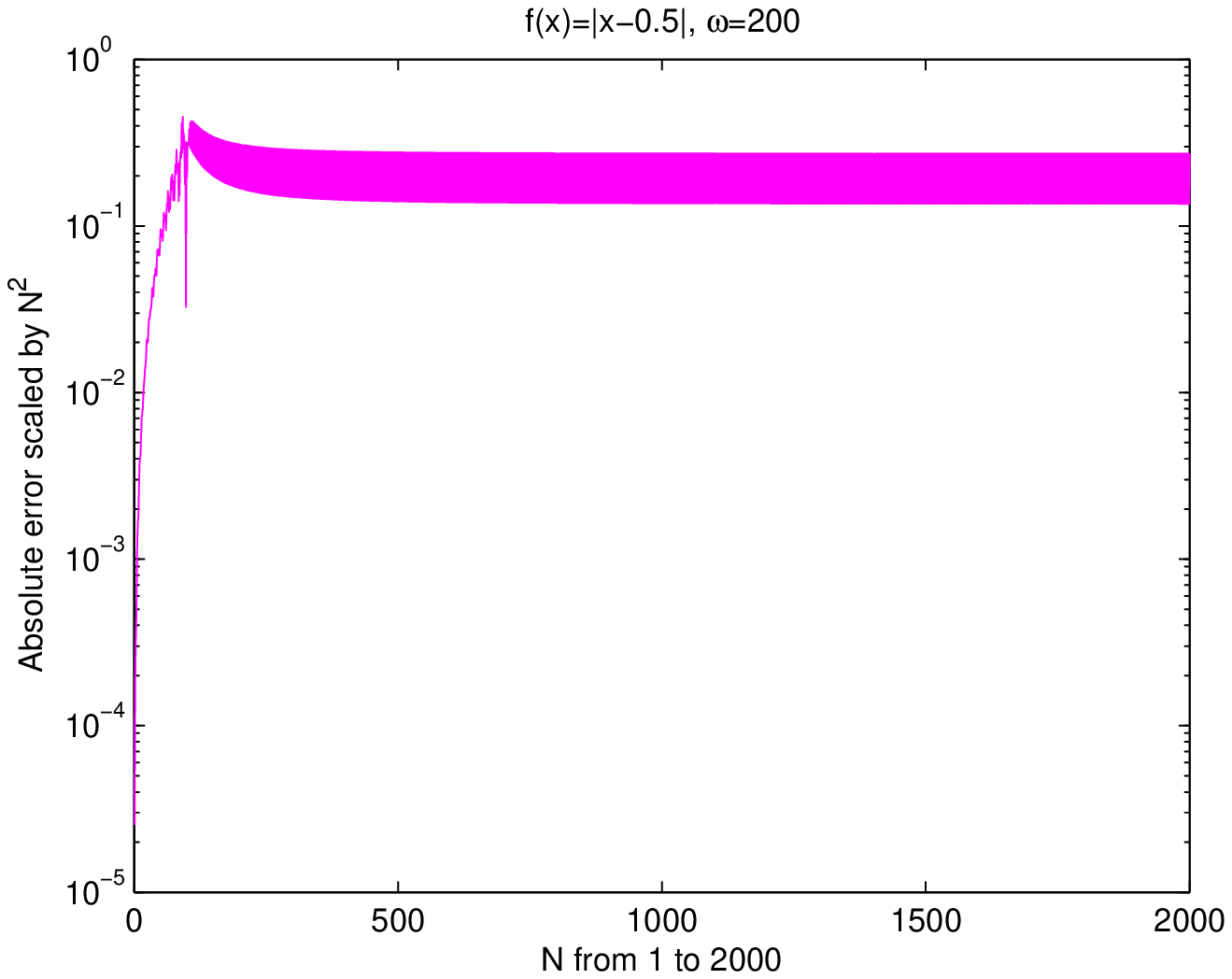}
\includegraphics[scale=0.42]{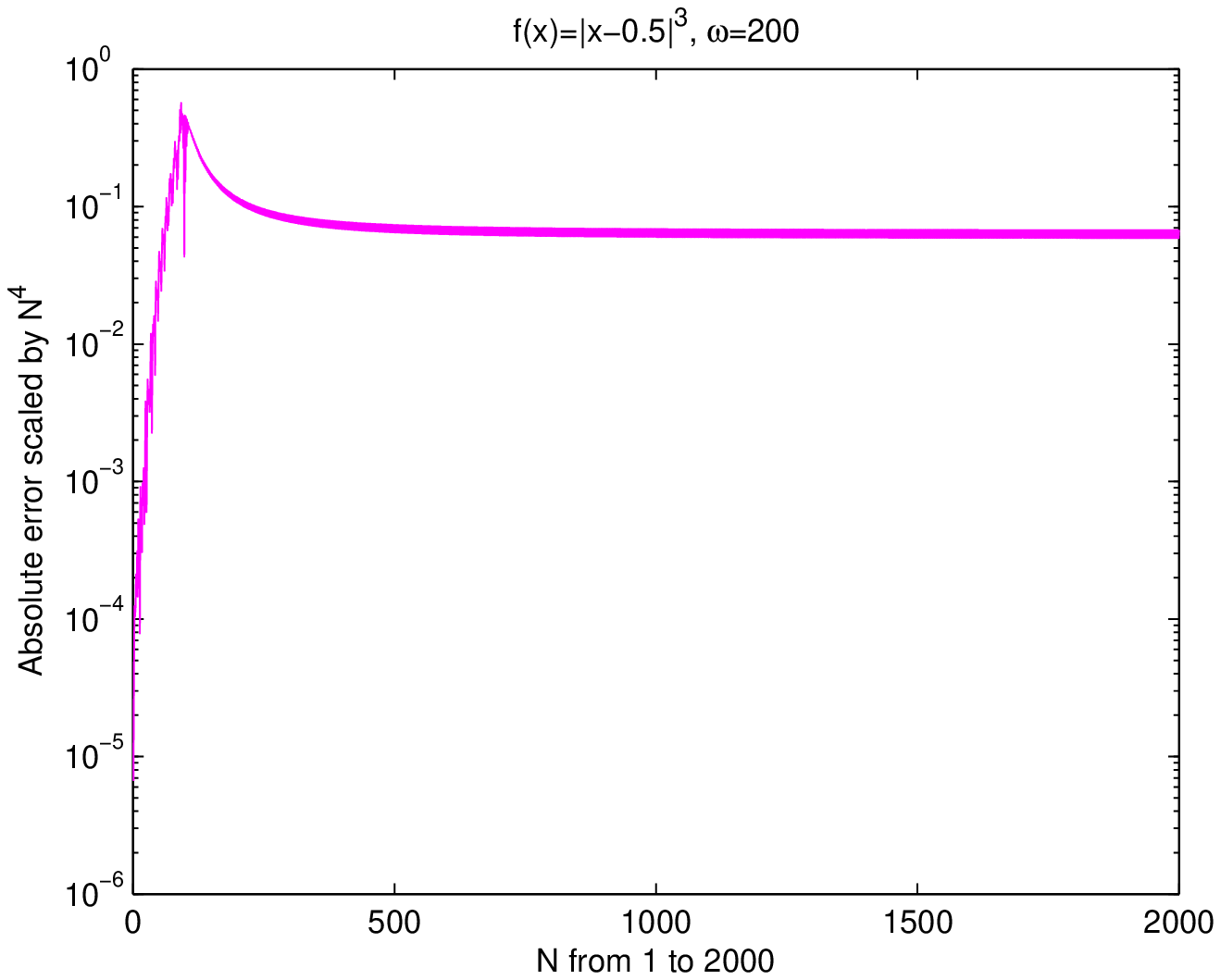}
\caption{Absolute errors scaled by
$N^2$ (left) and $ N^4$ (right) for the integral (\ref{eq:d1a})
for the Clenshaw-Curtis-Filon method with
$\omega=200$, $\alpha=0.2, \beta=0.4$, $N$ from $1$ to $2000$.}
\label{fig:2}
\end{center}
\end{figure}

\begin{figure}[hbtp]
\begin{center}
\includegraphics[scale=0.42]{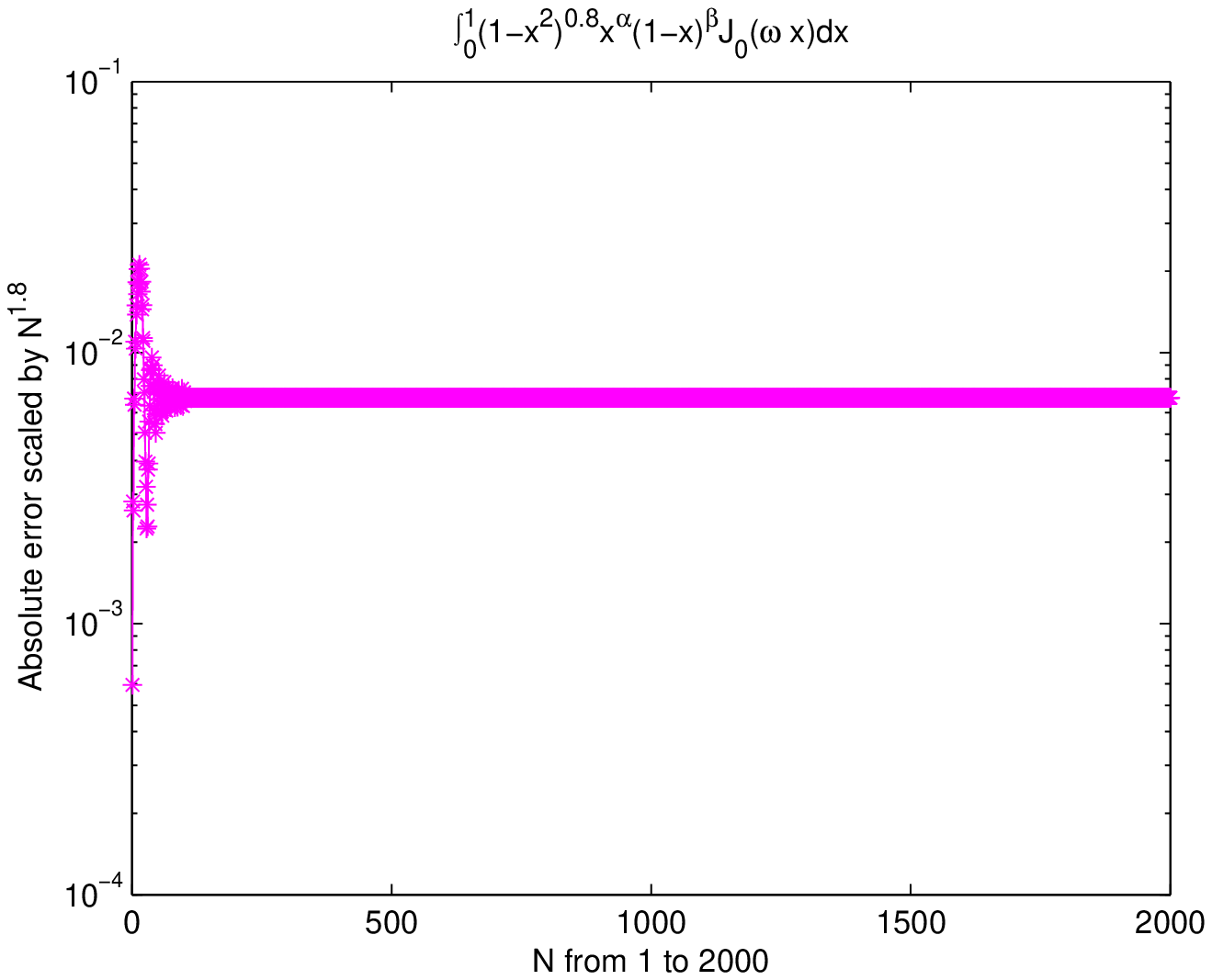}
\includegraphics[scale=0.42]{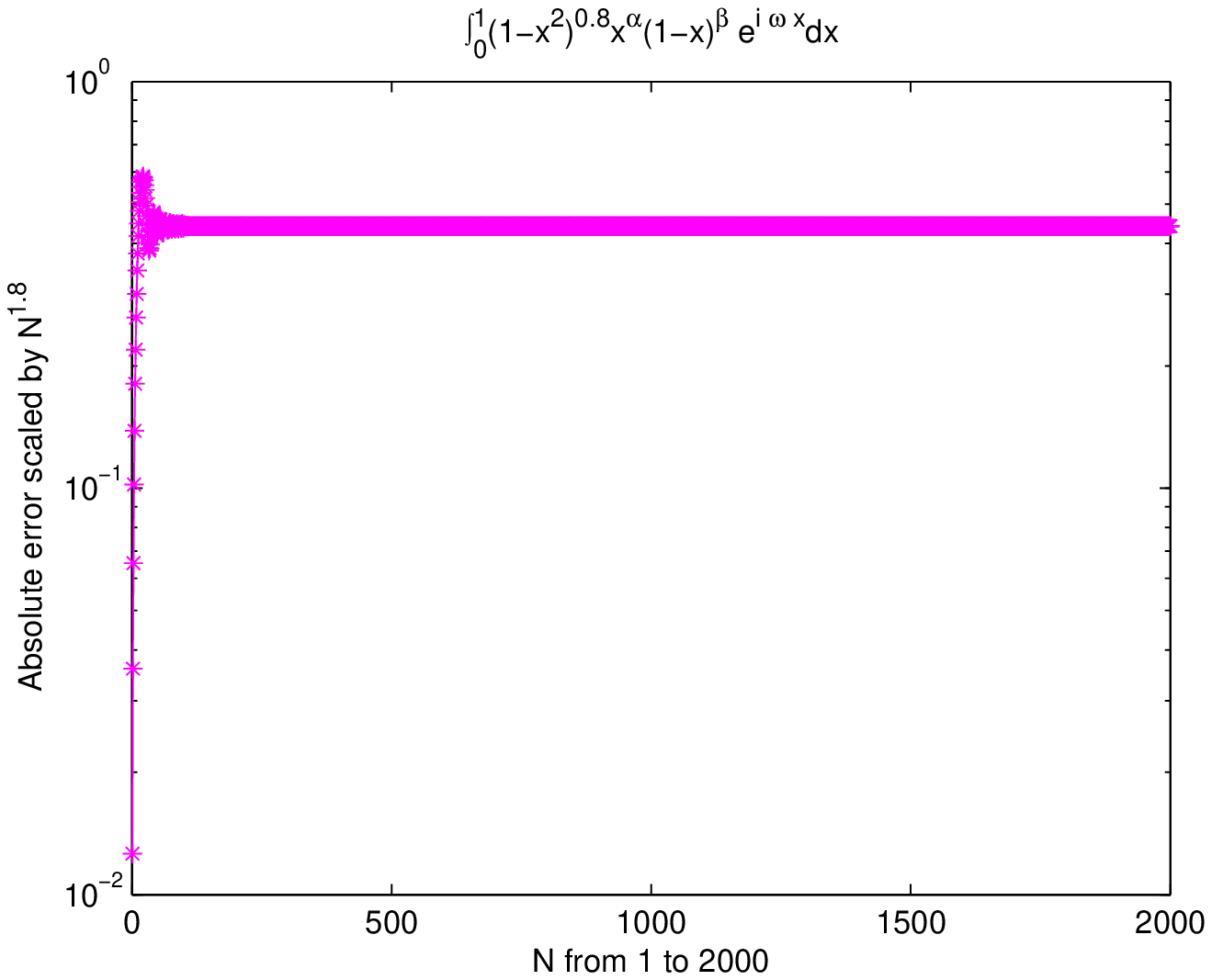}
\caption{Absolute  errors scaled by
$N^{1.8}$  for the integral (\ref{eq:d1b}) (left) and integral (\ref{eq:d1c}) (right)
for the Clenshaw-Curtis-Filon method with
$\omega=200$, $\alpha=-0.8, \beta=-0.9$, $N$ from $1$ to $2000$.}
\label{fig:2a}
\end{center}
\end{figure}

From above two examples and Theorem \ref{th:2}, we can see that the
Clenshaw-Curtis-Filon method is very efficient for the integral (\ref{eq:a1}),
and the accuracy can be improved by adding the number of the interpolation nodes.
Also, the accuracy increases as
$\omega$ increases.

\section{Conclusion}
In this paper, we present a Clenshaw-Curtis-Filon method for integral
(\ref{eq:a1}). The method is based on  a special Lagrange interpolation polynomial at $N+1$ Clenshaw-Curtis points and can be
computed by using $O(N\log N)$ operations. We first give a recurrence relation for the modified moments and present
an efficient algorithm for the moments. Then, we show that the proposed  method is uniformly convergent in
$N$ for fixed $\omega$. The numerical examples illustrate the efficiency and accuracy for this  method.
Moreover, the error bound (\ref{eq:b25}) is optimal on $N$ for  fixed $\omega$.
Here, the word ``optimal''
means that the asymptotic order can be attainable by some functions. Additionally, this result also holds for the Clenshaw-Curtis-Filon
method for integral
$\int_a^b f(x)(x-a)^\alpha(b-x)^\beta e^{i\omega x}dx$.
It should be noted that the accuracy can be improved by adding
the number of the interpolation nodes.

\end{document}